\documentclass[12pt]{article}

\usepackage{amsthm,amssymb,amsmath}
\usepackage{epsfig,fullpage}

\title{Embedding Quartic Eulerian Digraphs on the Plane}

\author{Dan Archdeacon \\
      Dept. of Math. and Stat. \\
      University of Vermont \\
      Burlington, VT 05405 \\ USA
\and
      C. Paul Bonnington\thanks{The author
is thankful for support from the Marsden Fund (Grant Number UOA-825)
administered by the Royal Society of New Zealand. This grant also supported
the other two authors on visits to NZ.} \\
      School of Mathematical Sciences\\
      Monash University\\
      Clayton, VIC 3800, Australia\\
      {\tt paul.bonnington@monash.edu}
\and
  Bojan Mohar\thanks{Supported in part by the
  Research Grant P1--0297 of ARRS (Slovenia), by an NSERC Discovery Grant (Canada)
  and by the Canada Research Chair program.}~\thanks{On leave from:
  IMFM \& FMF, Department of Mathematics, University of Ljubljana, Ljubljana,
  Slovenia.}\\
  Department of Mathematics\\ Simon Fraser University\\
  Burnaby, B.C.~~V5A 1S6 \\
  {\tt mohar@sfu.ca}\\
  \ 
  \\
  \textsl{In honour of Dan Archdeacon.} 
}
\date{}

\newtheorem{theorem}{Theorem}[section]
\newtheorem{lemma}[theorem]{Lemma}

\newtheorem{proposition}[theorem]{Proposition}

\newtheorem{problem}[theorem]{Problem}

\newtheorem{claim}{Claim}

\newcommand\E{\mathcal E}
\newcommand\Z{\mathbb Z}

\renewcommand\P{\cal P}
\begin{document}

\maketitle

\begin{abstract}
Minimal obstructions for embedding 4-regular Eulerian digraphs on the plane are considered in relation to the partial order defined by the cycle removal operation. Their basic properties are provided and all obstructions with parallel arcs are classified.
\end{abstract}

\section{Introduction}\label{introduction}

An {\em Eulerian digraph} is a directed graph such that at each vertex the
in-degree equals the out-degree. We allow our digraphs to have {\em loops}
(edges $uu$, where each loop counts towards the in-degree and the out-degree), and
{\em parallel edges}, that is, two copies of an edge $uv$, or two anti-directed edges $uv$ and $vu$. Our digraphs are not necessarily connected despite the usual convention underlying Eulerian graphs.

An {\em embedding} of an Eulerian digraph in a surface is a (not necessarily cellular) embedding of the
underlying graph such that in- and out-edges alternate in the rotation at each
vertex; hence the restriction that the in-degree equals the out-degree. In particular,
an Eulerian digraph with an embedding on the plane is called {\em diplanar}.
This kind of embedding for a digraph is very natural and was considered earlier in various contexts:
Andersen et al.~\cite{ABJ96} were motivated by questions about Eulerian trails with forbidden transitions; Bonnington et al.\ \cite{BCMM02,BHS04} and others \cite{CGH14,HLZX09} introduced digraph embeddings in the context of topological graph theory; Johnson \cite{Jo02} and Farr \cite{Farr13} explored different relations to the theory of graph minors in this context.
Other authors have considered different ways to embed directed graphs. For
example, Sneddon \cite{Sn04} studied ``clustered'' planar embeddings of digraphs where, at each vertex, all of the in-arcs occur sequentially in the local rotation. In \cite{Sn04} and \cite{BS11}, three different variations of minors are presented, each of which produces a finite set of obstructions to clustered planarity.
Clustered upward embeddings on the plane (where all edges are pointed ``upwards'') were considered in relation to graph drawing by Hashemi \cite{Ha01}.

Each face of an embedded Eulerian digraph is bounded by a directed cycle.
If the surface is orientable, then the faces fall into
two classes: those whose boundary cycle is clockwise and those whose boundary
cycle is anti-clockwise. It follows that the dual is bipartite. This is
not true for embeddings on a non-orientable surface, where the duals are necessarily non-bipartite.

A natural partial ordering on the set of all Eulerian digraphs is that of
{\em containment}: we say that $G$ contains an Eulerian digraph $H$ if $H$ is isomorphic to a subdigraph of $G$. Note that this is equivalent to saying that we can form $H$ from $G$ by a sequence
of removing directed cycles and removing isolated vertices. Note that removing directed cycles keeps us in
the class of Eulerian digraphs, which is why we allow disconnected graphs.

\begin{figure}[htb]
   \centering
   \includegraphics[width=5.2cm]{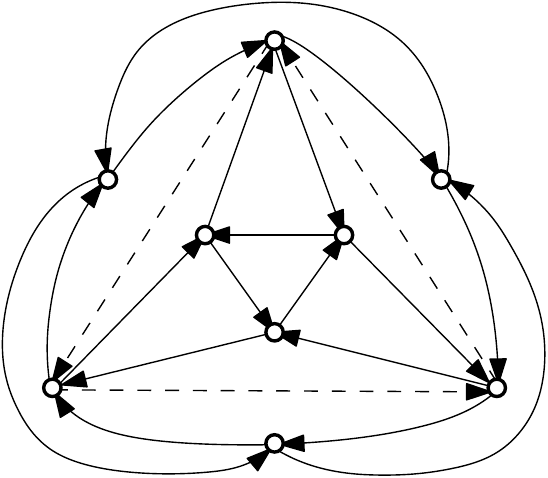}
   \caption{Removing directed cycles does not preserve diplanarity}
   \label{fig:1}
\end{figure}

We are interested in the class of digraphs that embed on a fixed surface.
One difficulty with the containment partial ordering is that embedding on
a surface is not hereditary under this order. An example is given in Figure~\ref{fig:1}; the removal of the dashed triangle gives a digraph which is not diplanar.

But all is not lost. Let us restrict our attention to Eulerian digraphs
with maximum in- and out-degree at most 2. The degree of each vertex
(the sum of the in-degree and out-degree) is hence either 0, 2, or 4. We will
usually delete isolated vertices, and {\em suppress} vertices $v$ of degree 2 by
replacing the directed arcs $uv,vw$ with a single arc $uw$. Hence we assume
our digraphs are regular of degree 4, so called {\em quartic}
Eulerian digraphs. For notational purposes, we will denote the class of all quartic Eulerian digraphs by $\E_4$.
For quartic Eulerian graphs we consider the partial order $(\E_4,\preceq)$
where we remove directed cycles, followed by suppressing any degree 2 vertices.

Quartic Eulerian digraphs are a very natural class of graphs to consider, as
they provide the simplest non-trivial way of studying embeddings of directed
graphs. In this sense they are the equivalent of studying embeddings of cubic undirected
graphs. They also appear as medial graphs of embeddings of undirected graphs in orientable surfaces (see a discussion later in this section).

The following lemma shows that embedding on a fixed surface is hereditary for the
class of quartic Eulerian digraphs in $\E_4$ with respect to the partial order $\preceq$.

\begin{lemma}
Let $G\in \E_4$ be a quartic Eulerian digraph. Suppose that $G$ embeds on a surface $S$. If $H \preceq G$, then $H$ also embeds on $S$.
\end{lemma}

\begin{proof}
Consider the embedding of $G$. By removing the edges of a cycle $C$ in $G$, the in-out property is preserved on the embedding of $G-E(C)$. Clearly, suppressing vertices of degree 2 also preserves embeddability. Since $H$ is obtained from $G$ by a sequence of such operations, it follows that $H$ is also embeddable in $S$.
\end{proof}

Whenever a property $\P$ is hereditary for a finite poset of digraphs, it is natural to consider minimal elements that do not have property $\P$. These are called \emph{minimal excluded digraphs} or \emph{obstructions} for the property $\P$:
these are digraphs that do not have property $\P$, but any strictly smaller digraph has.
Excluding these obstructions gives a characterization of digraphs with the given property.

The main focus of this paper is a partial progress towards the following goal.

\begin{problem}
Determine the complete set of obstructions for diplanar quartic Eulerian digraphs.
\end{problem}

Bonnington, Hartsfield, and \v Sir\'a\v n \cite{BHS04} examined a similar problem for embedding (not necessarily quartic) Eulerian digraphs.
Their embeddings also required that the in- and
out-edges alternate around a vertex. The difference between our results and
theirs occurs both in the class of graphs considered and the partial order used:  they allowed
Eulerian digraphs of arbitrarily large degree, and the partial order allowed
the directed version of arc-contractions. They gave a characterization of
the minimal non-planar digraphs under their partial order. They used this
partial order precisely because the property of embedding on a surface is not
preserved under removing directed cycles for digraphs with maximum degree exceeding 4.
The combination of a different partial order and a more restricted class of graphs
make the problems considered in \cite{BHS04} and those considered here quite different. This
is reflected in the different obstruction sets.
Furthermore, no known method exists that allows one result to derive from the other.

Another partial order, obtained by splitting vertices of degree 4 into two vertices of degree 2, has been considered as well (see \cite{ADHM16} and a thesis of Johnson \cite{Jo02}).

Returning to the relationship with embedded undirected graphs, our central problem can also be formulated as follows. Given a graph $G$ embedded in
a surface $S$, the {\em medial graph} is the graph $M$ whose vertices are the edges of $G$, and two vertices of $M$ are adjacent whenever
the corresponding edges are consecutive in a face of the embedded $G$. Thus, the medial graph
is 4-regular. When $S$ is orientable, we can direct $M$ by directing the
facial cycles, and then placing the induced directions on the edges of the medial
graph. Under this orientation, $M$ becomes a quartic Eulerian digraph. Our problem is
equivalent to finding the minimum genus graph $G$ whose directed medial graph $M(G)$ is
our given $D$. In particular, our problem is to characterize those digraphs $M$ which
are the directed medial graphs of planar graphs.

We give some preliminary lemmas in Section \ref{basics}, present the known list of
obstructions in Section \ref{obstructions}, and  indicate directions for future research in
Section \ref{conclusion}.

\section{Preliminary Results}\label{basics}

We first establish some terminology. It will be convenient to
distinguish when the underlying undirected graph is {\em planar}, and when
the directed graph is diplanar as defined in the introduction.
A pair of edges $uv,vu$ will be called a {\em digon}. A pair of {\em parallel edges} $uv,uv$ will be called an {\em anti-digon}.

For the convenience of the reader, we state some basic properties that will be used in the proofs. Let $H$ be an Eulerian digraph. Then the following properties clearly hold.
\begin{itemize}
\item[(i)]
For every partition $(A,B)$ of $V(H)$, the number of edges in the cut from $A$ to $B$ is the same as the number of edges in the cut from $B$ to $A$.
\item[(ii)]
$E(H)$ can be partitioned into directed cycles. A particular consequence of this is that if $H$ is not a directed cycle and $xy\in E(H)$, then $H-xy$ contains a directed cycle.
\item[(iii)]
If $xy$ is an edge in $H$ and a digraph $H'$ is obtained from $H-xy$ by removing some directed cycles, then $H'$ contains a directed path from $y$ to $x$.
\end{itemize}

In this section we give some results that may help us to focus on the
underlying problem. We start with some simple facts about quartic obstructions.

\begin{lemma}
\label{lem:basic}
Let\/ $G$ be a minimal non-diplanar digraph in $\E_4$. Then $G$ has the following properties.

{\rm (a)} The underlying multigraph $\hat G$ has no loops, and has at most two undirected
edges joining any two vertices. Moreover, $\hat G$ is 4-edge-connected.

{\rm (b)} $G$ is strongly $2$-edge-connected, i.e., for any two pairs of vertices $u_1,u_2$ and $v_1,v_2$, there are edge-disjoint directed paths $P_1,P_2$, where $P_i$ starts at $u_i$ for $i=1,2$, and one of the paths ends at $v_1$ and the other one at $v_2$.
\end{lemma}

\begin{figure}[htb]
   \centering
   \includegraphics[width=12cm]{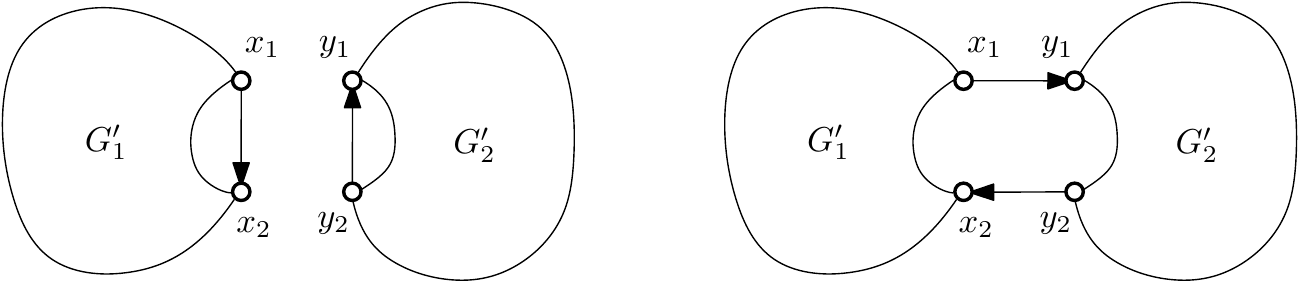}
   \caption{Dealing with 2-edge-cuts}
   \label{fig:2}
\end{figure}

\begin{proof}
The first claim in (a) is easy and the details are left to the reader. The second claim in (a) follows from (b), whose proof we discuss next.

Considering part (b), it is easy to prove that $G$ must be connected. Suppose that it is not strongly 2-edge-connected. By Menger's theorem there is an edge-cut consisting of fewer than four arcs. Since $G$ is Eulerian, the cut has precisely two arcs $e_1=x_1y_1$ and $e_2=y_2x_2$, one in each direction. Let us remove these arcs and form digraphs $G_1,G_2$ as follows. The digraph $G-e_1-e_2$ consists of two components $G_1'$ and $G_2'$, and we may assume that $x_1,x_2\in V(G_1')$ and $y_1,y_2\in V(G_2')$. We then set $G_1 = G_1' + x_1x_2$ and $G_2 = G_2' + y_2y_1$.
By using property (ii) from above, it is easy to see that $G_1\prec G$ and $G_2\prec G$. By the minimality of $G$, digraphs $G_1$ and $G_2$ can be embedded on the plane. We may assume that the added edges are on the boundary of the infinite face oriented differently. Figure \ref{fig:2} shows how these embeddings can be combined to obtain an embedding of $G$, thus yielding a contradiction.
\end{proof}

\begin{figure}[htb]
   \centering
   \includegraphics[width=10cm]{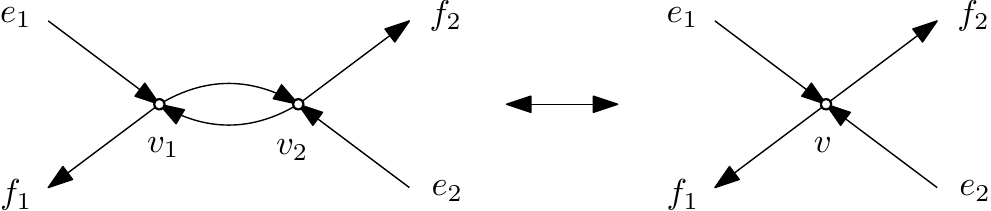}
   \caption{Contracting a digon}
   \label{fig:3}
\end{figure}

Let $G\in \E_4$ be a quartic Eulerian digraph and let $D$ be a digon. Define the {\em quotient}
$G/D$ to be the quartic digraph formed by identifying vertices $u$ and $v$ and deleting the two arcs in $D$.
We also say that $G/D$ is obtained from $G$ by \emph{contracting the digon} $D$.

\begin{lemma}
\label{lem:contractdigon}
Suppose that $G\in \E_4$ contains a digon $D$, and let $H = G / D$.

{\rm (a)} If $G$ is a diplanar obstruction, then $H$ is also a diplanar obstruction.

{\rm (b)} If $H$ is a diplanar obstruction, then $G$ is a diplanar obstruction if and only if $G-E(D)$ is diplanar.
\end{lemma}

\begin{proof}
Let $v_1,v_2$ be the vertices of $D$, and let $e_i,f_i$ be the edges incident with $v_i$ in $G-E(D)$ for $i=1,2$, as shown in Figure~\ref{fig:3}.

(a) Let us first show that $H$ is not diplanar. If it were, we could split the vertex $v$ obtained in contracting $D$ (since $e_1,f_1$ are consecutive in the local rotation around $v$) and then one could add the digon $D$ so that a diplanar embedding of $G$ would be obtained. (See Figure \ref{fig:3} moving right to left.) To show that $H$ is a diplanar obstruction it therefore suffices to see that for every cycle $C$ in $H$, $H-E(C)$ is diplanar. If $C$ does not pass through $v$, then we can use the planar embedding of $G-E(C)$. In this embedding, the digon $D$ bounds a face and thus it is easy to change it so that an embedding of $H-E(C)$ is obtained. (See Figure~\ref{fig:3} moving left to right.) If $C$ uses the edges $e_1$ and $f_1$ (or $e_2$ and $f_2$) then there is nothing to prove since in that case diplanarity of $G-E(C)$ implies diplanarity of $H-E(C)$ with an added loop at $v$. Finally, if $C$ uses the edges $e_1$ and $f_2$ (say), then we first embed $G-E(C)-v_1v_2$; by contracting the edge $v_2v_1$ of $D$, we thus obtain an embedding of $H-E(C)$. This shows that $H$ is a diplanar obstruction.

(b) Since $H$ is a diplanar obstruction, we see as above that $G$ is not diplanar. Thus, it suffices to see that $G-E(C)$ is diplanar for every cycle $C$ in $G$. The proof is similar to that in part (a) except that now we use embeddings of $H-E(C)$ to obtain embeddings of $G-E(C)$. We omit the details. The only added ingredient is that $G-E(D)$ also needs to be diplanar, which is guaranteed by the condition in the statement.
\end{proof}

By Lemma \ref{lem:contractdigon}, it suffices to find all diplanar obstructions without digons. Those that have digons, can be obtained from these by ``splitting vertices'' and adding digons (that is, the reverse operation to contraction of a digon). All we need to check is that after any such splitting we obtain a diplanar digraph prior to inserting the digon. Each vertex of $H\in\E_4$ can be split in two ways and then a digon joining the two resulting vertices of degree 2 can be added. We say that splitting of a vertex is \emph{admissible} if the digraph obtained after the splitting is diplanar. If $v$ is split into vertices $v^1,v^2$ of degree 2, and $p\ge1$ is an integer, we can \emph{add a path of $p$ digons} by adding vertices $x_1,\dots,x_{p-1}$ and digons between $x_i$ and $x_{i-1}$ for $i=1,\dots, p$, where $x_0=v^1$ and $x_p=v^2$. It is clear (by admissibility of the splitting and by Lemma \ref{lem:contractdigon}(b)) that this always gives a diplanar obstruction.

The following result gives the complete description of diplanar obstructions containing digons.

\begin{theorem}
\label{thm:insertdigons}
Let $H$ be a minimal non-diplanar graph in $\E_4$. Let $\{v_1,\dots,v_s\}$ be a set of $s\ge1$ vertices of $H$. For $i=1,\dots, s$, consider an admissible splitting of $v_i$ resulting into two vertices $v_i^1,v_i^2$ of degree $2$ and add a path of $p_i\ge1$ digons. Then the resulting digraph is a diplanar obstruction. Conversely, every diplanar obstruction which gives rise to $H$ after contracting a set of digons can be obtained from $H$ in this way.
\end{theorem}

\begin{proof}
Adding one digon to any admissible splitting gives rise to a diplanar obstruction. After adding the digon, all previous splittings keep their admissibility. The two new vertices have the property that one of the splittings is not admissible since it gives a digraph isomorphic to the original obstruction, while the other one is admissible. By using these new admissible splittings, all we achieve is to extend a digon to a path of digons. This yields the theorem.
\end{proof}

\section{The Obstructions}\label{obstructions}

In this section we give the current set of known diplanar obstructions.

\subsection{Doubled cycles}
\label{subsect:doubledcycles}

The doubled cycle $\overrightarrow{C}_n^{(2)}$ ($n\ge3$) is formed
by replacing each edge $(i,i+1)$ ($i=1,\dots, n$, summation modulo $n$) in a directed cycle $\overrightarrow C_n$ with two
directed edges in parallel (that is, an anti-digon). It is not hard to show that these graphs are non-diplanar
and that they are minimal since removing any directed cycle leaves a digraph of maximum
degree 2.

\begin{figure}[htb]
   \centering
   \includegraphics[width=9cm]{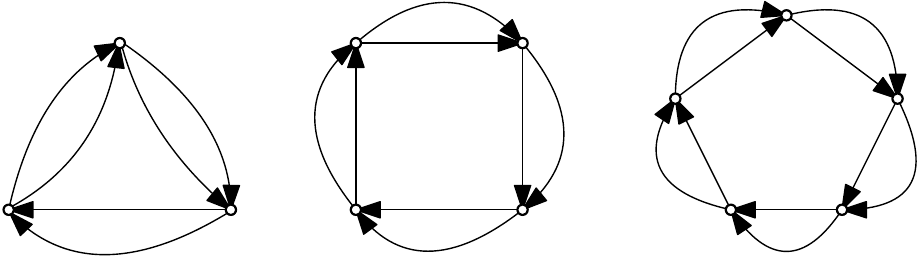}
   \caption{Doubled cycles of lengths 3, 4, and 5}
   \label{fig:4}
\end{figure}

By splitting a vertex of $\overrightarrow{C}_n^{(2)}$ (and suppressing the resulting vertices of degree 2), we obtain $\overrightarrow{C}_{n-1}^{(2)}$.
Thus, Lemma \ref{lem:contractdigon}(b) implies that we cannot obtain further diplanar obstructions by adding digons to $\overrightarrow{C}_n^{(2)}$ when $n\ge4$. The exception is when $n=3$. In that case, we can split one, two or three vertices and obtain diplanar obstructions shown in Figure \ref{fig:5}, where each digon can be replaced by a path of digons. Note that adding all three digons can be done in two different ways. One gives the 3-prism $P_3^+$, the other one the M\"obius ladder $M_3^+$.

\begin{figure}[htb]
   \centering
   \includegraphics[width=11.6cm]{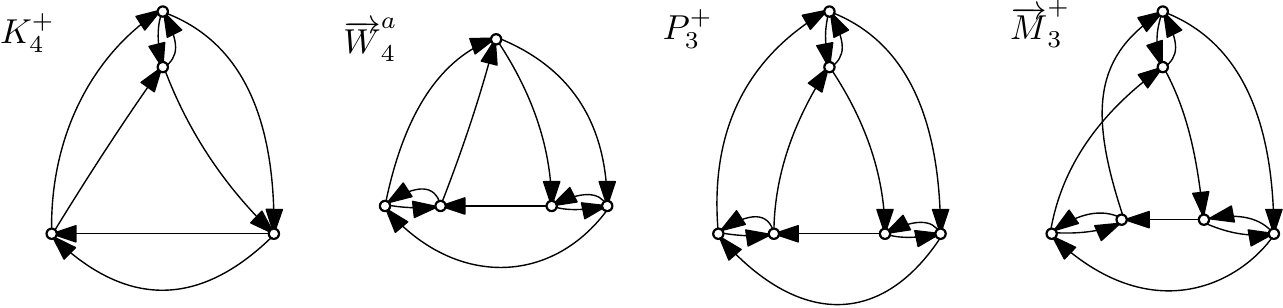}
   \caption{Adding digons to the doubled cycle of length 3}
   \label{fig:5}
\end{figure}

\subsection{Circulants and M\"obius ladders}

Consider a directed cycle $(1,2,\dots,2n)$, where $n\ge3$ is odd. For each even $i$, add the arc $(i,i+n)$, and
for each odd $i$, add the arc $(i,i-1)$. This gives a digraph $\overrightarrow M_n^+$ called the \emph{M\"obius ladder} since it can be obtained from the usual (undirected) M\"obius ladder with
$n$ spokes by replacing every other rim edge with a digon. The digraphs for $n=3,5$ in Figure \ref{fig:6} are shown as embedded in the projective plane;  the generalization is obvious.

\begin{figure}[htb]
   \centering
   \includegraphics[width=10cm]{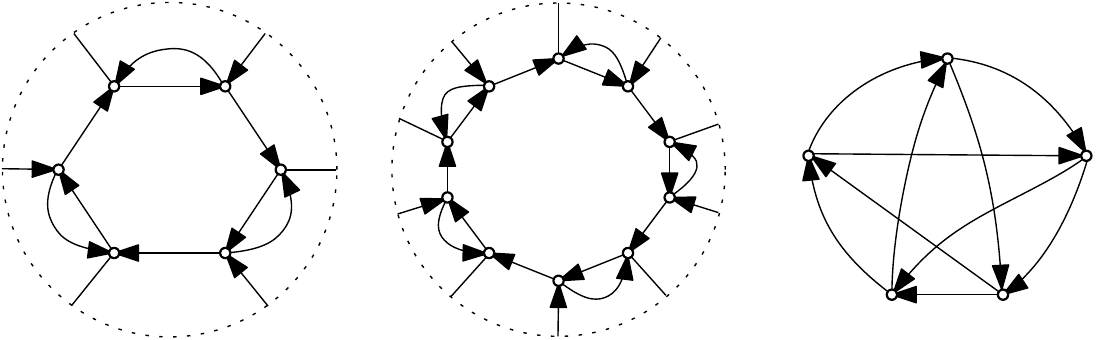}
   \caption{M\"obius ladders in the projective plane and the contraction $Z_5$}
   \label{fig:6}
\end{figure}

Contracting all 5 digons in $\overrightarrow M_5^+$ gives an example of a diplanar obstruction based on the orientation of $K_5$ shown in Figure \ref{fig:6}, where it is denoted as $Z_5$. In general, contracting all $n$ digons
in $\overrightarrow M_n^+$ gives a diplanar obstruction based on the Cayley digraph $Z_n$ with group $\Z_n$ ($n=2k+1$) using the
generating set $\{1, k\}$. The following proposition follows by considering the canonical directed embedding of $Z_n$ in the projective plane, and noting that any directed cycle with fewer than $n$ vertices will be ``essential'' in that embedding, thus yielding a directed planar embedding upon removal.

\begin{proposition}
For $n\ge5$, $Z_n$ is a diplanar obstruction.
\end{proposition}

For $n\ge5$, each vertex of the diplanar obstruction $Z_n$ has only one admissible splitting; that is, the one used to obtain $\overrightarrow M_n^+$. The other splitting gives rise to a digon, whose contraction yields $Z_{n-2}$ which is not diplanar. It follows that $\overrightarrow M_n^+$ is a diplanar obstruction.
Moreover, all diplanar obstructions whose digon contractions yield $Z_n$ are obtained from $\overrightarrow M_n^+$ by contracting some digons and replacing some of them by paths of digons of greater length.

Finally, we note that the obstruction $Z_3$ is isomorphic to $\overrightarrow{C}_3^{(2)}$ and the admissible splittings were discussed in Section~\ref{subsect:doubledcycles}.

\subsection{Two simple sporadic examples}\label{secsporadic}

Two sporadic examples without digons and anti-digons, $\overrightarrow K_{2,2,2}$ and $\overrightarrow K_{4,4}$, are shown in Figure \ref{fig:7}. These examples were found by looking at small 4-regular graphs.
Checking that these are diplanar obstructions is fairly easy, since the number of vertices is small and the order of the automorphism
group is large.

\begin{figure}[htb]
   \centering
   \includegraphics[width=7.2cm]{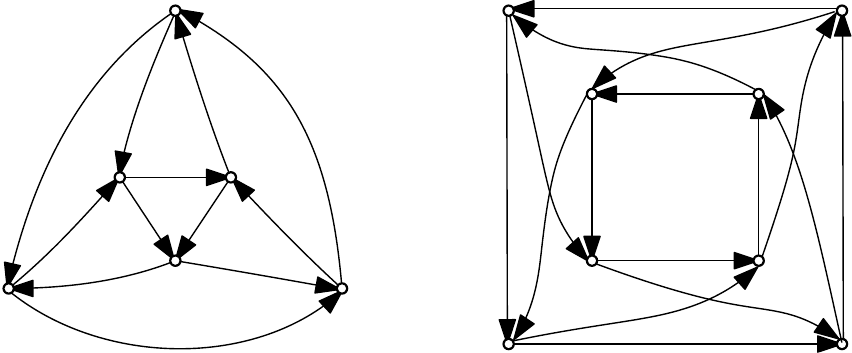}
   \caption{Non-diplanar orientations of $K_{2,2,2}$ and $K_{4,4}$}
   \label{fig:7}
\end{figure}

The octahedron $K_{2,2,2}$ is planar but its orientation $\overrightarrow K_{2,2,2}$ does not have its faces bounded by directed triangles. Since the octahedron has (essentially) a unique embedding on the plane, this digraph is not diplanar. There are many non-diplanar orientations of this graph, but up to symmetries, this is the only orientation of $K_{2,2,2}$ that gives a diplanar obstruction for diplanarity. Any other orientation can be obtained from the planar one by reversing orientations of edges of an Eulerian subgraph, for which we may assume that it has at most 6 edges. Changing orientation of a triangle is easy to exclude (removing the ``opposite triangle" leaves $\overrightarrow C_3^{(2)}$). The same holds if the orientation on two disjoint triangles is switched. The only directed cycles besides facial triangles are hamilton cycles, all of which are isomorphic to each other. But switching their orientation is the same as switching the orientation on two disjoint triangles. The only remaining possibility is to switch the edges on two triangles sharing a vertex. This gives $\overrightarrow K_{2,2,2}$. For this orientation, there are only two directed triangles (those used in switching); removing one of them gives a diplanar digraph consisting of three digons. Removing a cycle of any larger length leaves at most two vertices of degree 4, which is necessarily diplanar.

The other example is even easier. Since $K_{4,4}$ is not planar, $\overrightarrow K_{4,4}$ cannot be diplanar. Up to symmetries, $\overrightarrow K_{4,4}$ has only two different directed cycles, a 4-cycle and an 8-cycle, and their removal leads to diplanar digraphs.

\subsection{Obstructions containing anti-digons}

Two further examples of diplanar obstructions with anti-digons are shown in Figure \ref{fig:8}.

\begin{figure}[htb]
   \centering
   \includegraphics[width=7cm]{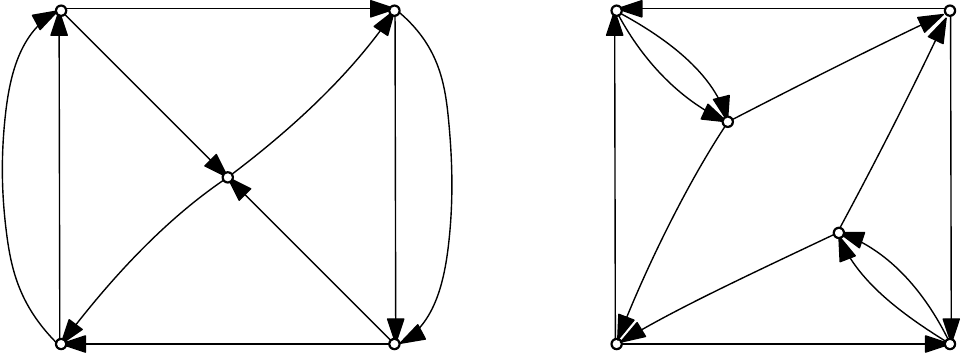}
   \caption{Two diplanar obstructions with anti-digons}
   \label{fig:8}
\end{figure}

These are just special cases of diplanar obstructions that can be obtained by taking a cyclic sequence of anti-digons as shown in Figure \ref{fig:9}(a) and (b), or combining building blocks of different lengths shown in Figure \ref{fig:9}(c). The first kind will be denoted by $\overrightarrow{L}_n$, where $n$ is the number of anti-digons; it will be called the \emph{anti-ladder} if $n$ is even and \emph{M\"obius anti-ladder} if $n$ is odd. The second kind is obtained by taking $p\ge 1$ copies of the digraph shown in Figure \ref{fig:9}(c), whose respective number of anti-digons are $n_1,n_2,\dots,n_p$ (and each $n_i\ge1$). Then the right vertex of each of these is identified with the left vertex of the next one (cyclically). The resulting digraph, denoted $\overrightarrow{N}(n_1,\dots,n_p)$, is clearly a diplanar obstruction. We observe that the diplanar obstructions in Figure \ref{fig:8} are $\overrightarrow{N}(2)$ and $\overrightarrow{N}(1,1)$, respectively. Furthermore, it is also worth observing that $\overrightarrow{C}_3^{(2)} = \overrightarrow{N}(1)$.

\begin{figure}[htb]
   \centering
   \includegraphics[width=11cm]{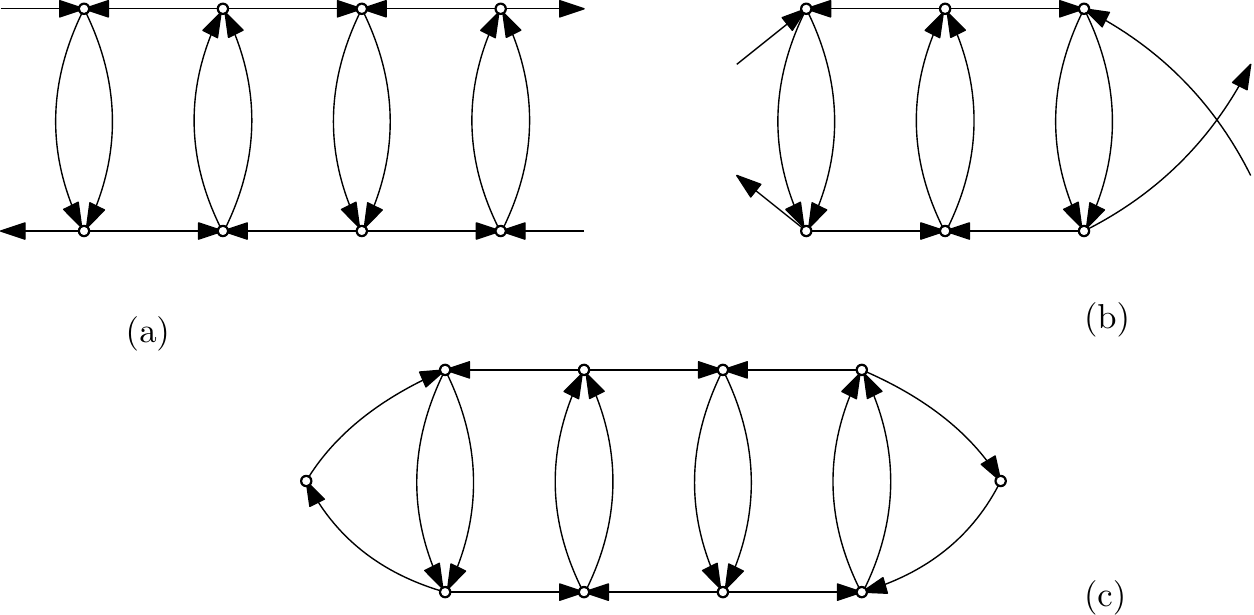}
   \caption{Building blocks for diplanar obstructions with anti-digons. (a) and (b) are ladders (even number of anti-digons) and M\"obius ladders (odd number of anti-digons), (c) shows a basic building block for the remaining diplanar obstructions.}
   \label{fig:9}
\end{figure}

\begin{theorem}
\label{thm:antidigons}
Suppose that $G$ is a diplanar obstruction in $\E_4$ that has no digons, but contains an anti-digon $D$.
If every cycle in $G$ intersects $D$, then $G$ is isomorphic to $\overrightarrow{C}_n^{(2)}$ for some $n\ge3$.
Otherwise $G$ is either a $($M\"obius$)$ anti-ladder $\overrightarrow{L}_n$ $(n\ge2)$ or is isomorphic to $\overrightarrow{N}(n_1,\dots,n_p)$, where $p$ and $n_1,\dots,n_p$ are positive integers.
\end{theorem}

\begin{proof}
Let $u,v$ be the vertices of $D$ and suppose the parallel edges in $D$ are oriented from $u$ to $v$. Let $u_1,u_2$ and $v_1,v_2$ be the in-neighbors of $u$ and the out-neighbors of $v$, respectively. By Lemma \ref{lem:basic}(a), these vertices are distinct from $u$ and $v$. By Lemma \ref{lem:basic}(b), $G$ contains edge-disjoint paths $P_1,P_2$ from $\{v_1,v_2\}$ to $\{u_1,u_2\}$ and by adjusting notation, we may assume that $P_i$ joins $v_i$ and $u_i$, $i=1,2$. Let $Q_i = P_i + u_iu + uv + vv_i$.

Suppose first that every cycle in $G$ intersects $D$. In that case $G-E(Q_1\cup Q_2)$ has no edges, meaning that $G= Q_1\cup Q_2$. The paths $P_1$ and $P_2$ must have a vertex in common since $G$ is not diplanar. Following the path $P_2$, we see that its intersections with $P_1$ form a sequence of vertices whose order on $P_1$ is in the direction from $v_1$ towards $u_1$ (otherwise, there would be a cycle in $G-D$). Since there are no vertices of degree 2 in $G$, this gives that $G$ is isomorphic to $\overrightarrow{C}_n^{(2)}$, where $n\ge3$.

Suppose now that $G$ has a cycle $C$ disjoint from $D$. By removing $E(C)$, we obtain a diplanar graph. We may assume that $P_1$ and $P_2$ are edge-disjoint from $C$. In the diplanar embedding of $G-E(C)$, one of the paths must be embedded in the interior of the disk bounded by the anti-digon $D$, and the other path in the exterior. Let $B_i$ be the component of $G-E(C)$ containing $P_i$. Note that $C$ contains a path from $B_1$ to $B_2$ and a path from $B_2$ to $B_1$. This implies that there are no other components beside $B_1$ and $B_2$ since the removal of a cycle contained in such a component would give a non-diplanar digraph.

We can take a $(v_1,u_1)$-trail $Q_1$ in $B_1$ and a $(v_2,u_2)$-trail $Q_2$ in $B_2$. We say that the triple $(C,Q_1,Q_2)$ is a \emph{connector} in $G':=G-\{u,v\}$ if $C$ has a vertex in common with $Q_1$ and has a vertex in common with $Q_2$. A connector exists for every cycle $C$ in $G'$ -- we obtain one by taking $Q_i$ to be an Eulerian trail in $B_i$ for $i=1,2$. The connector is \emph{full} if $E(Q_i)=E(B_i)$ for $i=1,2$. A basic observation about connectors is that $D$ together with the edges in the connector is not diplanar. This implies the following property.

\begin{claim}
\label{cl:1}
Every connector in $G'$ is full.
\end{claim}

\begin{proof}
Let $H=G'-E(C\cup Q_1\cup Q_2)$. Observe that $H$ is Eulerian.
If $(C,Q_1,Q_2)$ is not full, there is a cycle in $H$. By removing that cycle from $G$, a non-diplanar digraph is obtained, which contradicts the property of the diplanar obstructions.
\end{proof}

Let $(C,Q_1,Q_2)$ be a connector. Let $v_1=x_1, x_2,\dots x_{p-1},x_p=u_1$ be the sequence of vertices on the trail $Q_1$.
We denote by $Q_1(x_i,x_j)$ the segment of $Q_1$ from $x_i$ to $x_j$ (with slight abuse of notation if the vertex $x_i$ or $x_j$ appears twice on $Q_1$, where $i$ and $j$ are clear from the context).

\begin{claim}
\label{cl:2}
If $x_i=x_j$, where $i<j$, then $V(C)\cap V(Q_1) \subseteq \{x_{i+1},\dots, x_{j-1}\}$.
\end{claim}

\begin{proof}
Since $x_i$ appears twice in $Q_1$, it is not on $C$. Suppose that $C$ passes through a vertex $x_k$ on $Q_1$, where $k<i$ or $k>j$. Replace $Q_1$ by the trail $Q_1' = Q_1(x_1,x_i) \cup Q(x_j,x_p)$. Then $(C,Q_1',Q_2)$ is still a connector, contradicting Claim \ref{cl:1}.
\end{proof}

Suppose that $x$ and $y$ are two vertices on $C$. They split the cycle in two directed paths, the \emph{$(x,y)$-segment} $C(x,y)$ from $x$ to $y$ and the \emph{$(y,x)$-segment} $C(y,x)$ from $y$ to $x$.

\begin{claim}
\label{cl:3}
Suppose that vertices $x_i$ and $x_j$ ($i<j$) on the trail $Q_1$ lie on $C$. Then $C(x_i,x_j)$ does not intersect $Q_2$.
\end{claim}

\begin{proof}
Let $C'$ be a cycle in $Q_1(x_i,x_j) \cup C(x_j,x_i)$. It is easy to see that $u_1,u_2,v_1,v_2$ are in the same connected component of $G'-E(C')$. This implies that $G-E(C')$ is not diplanar, a contradiction.
\end{proof}

\begin{claim}
\label{cl:4}
Suppose that vertices $x_i$ and $x_j$ ($i<j$) on the trail $Q_1$ lie on $C$. Then $x_i$ and $x_{i+1}$ form an anti-digon and one of the edges $x_ix_{i+1}$ is on $C$.
\end{claim}

\begin{proof}
We may assume that $j>i$ is smallest possible such that $x_j$ belongs to $C$. Our goal is to prove that $j=i+1$ and that $C(x_i,x_{i+1})=x_ix_{i+1}$. If $x_{i+1}$ is not on $C$, then Claim 1 implies that it appears twice on $Q_1$. However, the segment between these two appearances cannot contain both $x_i$ and $x_j$, which contradicts Claim \ref{cl:2}.

Therefore we know that $j=i+1$. Let $C'=C(x_{i+1},x_i)+x_ix_{i+1}$ and let $Q_1'$ be obtained from $Q_1$ by replacing the edge $x_ix_{i+1}$ by $C(x_i,x_{i+1})$. Claim \ref{cl:3} implies that $C(x_{i+1},x_i)$ intersects $Q_2$, thus $(C',Q_1',Q_2)$ is a connector. By what we proved above, the vertex $x'$ on $Q_1'$ following $x_i$ must be on $C'$. However, $x'$ was originally part of the cycle $C$ and by Claim \ref{cl:3}, two of its edges were on $Q_1$. Thus, $x'$ can be on $C'$ only if $x'=x_{i+1}$, which gives the conclusion of the claim.
\end{proof}

\begin{claim}
\label{cl:5}
If $C$ has more than one vertex in $Q_1$, then it has precisely two vertices that are consecutive on $Q_1$ and form an anti-digon in $G$. Moreover, one of the following cases occurs: either $v_1$ and $u_1$ form an anti-digon, or $v_1=u_1$. The same holds for $v_2$ and $u_2$.
\end{claim}

\begin{proof}
Claim \ref{cl:4} implies that the vertices in $Q_1\cap C$ form an interval on $Q_1$ and all edges on this interval are contained in anti-digons. If there is more than one anti-digon then the removal of a cycle in $Q_2\cup \{vv_2,u_2u,uv\}$ gives a digraph which is not diplanar. Thus, $C$ intersects $Q_1$ precisely in two consecutive vertices $x_i, x_{i+1}$. By using Claim \ref{cl:1} it is easy to infer that $Q_1(x_1,x_i)$ is a simple path (no repeated vertices) and so is $Q_1(x_{i+1},x_p)$. Each vertex on these two subpaths apart from $x_i$ and $x_{i+1}$ appears precisely twice, once on each subpath (since $G$ has no vertices of degree 2). If the two subpaths are disjoint, then there are no vertices apart from $x_i$ and $x_{i+1}$. This means that $i=1$ and $p=2$, and thus $Q_1=v_1u_1$ forms an anti-digon.
Otherwise, $v_1=x_1$ appears twice on $Q_1$. Suppose that $x_t=x_1=v_1$ where $i+1<t\le p$.
Consider a cycle $C'$ contained in $Q_1(x_1,x_t)$. We may take $C'$ so that it passes through $v_1$. There is a corresponding connector $(C',Q_1',Q_2')$ where $Q_1'=Q_1(x_t,x_p)$ and $Q_2'$ contains all edges of $Q_2\cup C$. The claims applied to this connector show that $C'$ has at most two vertices in common with $Q_1'$. If there are two, the proof above shows that $Q_1'=v_1u_1$ (forming an anti-digon) and, since $C'$ contains $x_1=x_t$, we have $t=p-1$.
On the other hand, if $C'$ intersects $Q_1'$ only in $x_1$, then there are no other vertices on $Q_1'$ and we have $t=p$ and thus $v_1=u_1$.

The proof for $v_2$ and $u_2$ is the same. This completes the proof of the claim.
\end{proof}

In the next claim we shall consider the case when $v_1=u_1$. In this case, we consider the connector $(C,Q_1,Q_2)$, where $Q_1$ is just the vertex $v_1=u_1$, $C$ is a cycle in $G'$ containing the two edges incident with $v_1$ in $G'$, and $Q_2$ is a $(v_2,u_2)$-trail in the rest of $G'$.

\begin{claim}
\label{cl:6}
If $v_1=u_1$ and the connector $(C,Q_1,Q_2)$ is as described above, then $C=v_1y_1y_2v_1$ is a $3$-cycle, and the vertices $y_1y_2$ form an anti-digon.
\end{claim}

\begin{proof}
It follows from previous claims that $C$ intersects $Q_2$ in at most two vertices and that $C$ has no other vertices apart from those on $Q_1$ and $Q_2$. Of course $C$ is not a digon (since we have excluded digons), thus it must have two vertices on $Q_2$. The claim now follows from Claim~\ref{cl:5}.
\end{proof}

After this preparation, we are able to complete the proof. Start with the digon $D$ and consider its neighbors $u_1,u_2,v_1,v_2$. If $v_1$ and $u_1$ form an anti-digon $D'$, we use Claim~\ref{cl:5} on $D'$ and continue doing this as long as we either come back to $D$ by taking the next and the next digon and so on, or we come to the situation that an out-neighbor and an in-neighbor of the anti-digon are the same vertex, call it $x$. In the latter case, the next neighbors of $x$ form another anti-digon by Claim~\ref{cl:6}. It is now evident that we obtain the structure as described by the theorem.
\end{proof}

\section{Conclusion}\label{conclusion}

We conclude with some pointers to further research. In this paper we presented all known diplanar obstructions. It is not known if our list is complete. We determined how to obtain all obstructions with digons from those that have none. Upon applying Theorem~\ref{thm:antidigons}, this would mean that we could turn our attention to characterising diplanar obstructions where the underlying graph is simple, and either a) planar, or b) non-planar. If the underlying graph is non-planar, then we could consider
when the underlying graph contains a M\"obius ladder $M_n$ for different values of $n$.
(Consideration of this particular family has proved useful in a similar problem for characterising planar induced subgraphs.) The remaining case would then be when the underlying graph
is simple, planar, and 3-connected, which the authors hope would be more straightforward.

\subsection*{Acknowledgements.}
This paper is based on a 2005 draft manuscript prepared by D.A. with our thoughts on planar digraphs. The results are based on concepts first discovered in Auckland in December, 2004.


\bibliographystyle{abbrv}
\bibliography{digraphemb}

\end{document}